\documentclass{amsart}

\usepackage{graphicx}

\newtheorem{theorem}{Theorem}
\newtheorem*{theorem*}{Theorem}
\newtheorem*{theorem-non}{Theorem}
\newtheorem*{proposition*}{Proposition}

\theoremstyle{definition}
\newtheorem*{definition}{Definition}

\newtheorem*{lemma*}{Lemma}

\newtheoremstyle{named}{}{}{\itshape}{}{\bfseries}{.}{.5em}{\thmnote{#3}}
\theoremstyle{named}
\newtheorem*{namedtheorem}{Lemma}

\begin{document}

\title{On fiber diameters of continuous maps}

\author{Peter S. Landweber}
\address{Department of Mathematics, Rutgers University, Piscataway, NJ 08854}
\email{landwebe@math.rutgers.edu}

\author{Emanuel A. Lazar}
\address{Department of Mathematics, University of Pennsylvania, Philadelphia, PA 19104}
\email{mlazar@math.upenn.edu}
\urladdr{www.math.upenn.edu/~mlazar}

\author{Neel Patel}
\address{Department of Mathematics, University of Pennsylvania, Philadelphia, PA 19104}
\email{neelpa@sas.upenn.edu}
\urladdr{www.math.upenn.edu/~neelpa}

\date{\today}

\begin{abstract}
We present a surprisingly short proof that for any continuous map $f : \mathbb{R}^n \rightarrow \mathbb{R}^m$, if $n>m$, then there exists no bound on the diameter of fibers of $f$.  Moreover, we show that when $m=1$, the union of small fibers of $f$ is bounded; when $m>1$, the union of small fibers need not be bounded.  Applications to data analysis are considered.
\end{abstract}

\maketitle

High-dimensional data sets are often difficult to analyze directly and, consequently, methods of simplifying them are important to modern data-intensive sciences.  Continuous mappings $f:\mathbb{R}^n \rightarrow \mathbb{R}^m$ are frequently used to reduce the dimension of large data sets.  Indeed, a classic result of Johnson and Lindenstrauss \cite{johnson1984extensions} shows that for $N$ points in any Euclidean space, there exists an injective Lipschitz function which maps these points into $\mathbb{R}^{O(\log N)}$ with minimal distortion in pairwise distances.  However, while continuous maps enjoy many desirable properties, the following suggests that a measure of caution should be exercised before employing them for high-dimensional data analysis.  We present a simple proof that for any continuous map $f : \mathbb{R}^n \rightarrow \mathbb{R}^m$, if $n>m$ then there exists no bound on the diameter of fibers of $f$.  Therefore, points can be arbitrarily far apart in $\mathbb{R}^n$, yet map to the same point under $f$.  

\begin{definition}
The \textbf{fibers} of a map $f:X \to Y$ are the preimages $f^{-1}(y) = \{x \in X : f(x)=y \}$ of points in $Y$.
\end{definition}
\begin{definition}
The \textbf{diameter} of a set $A$ is the supremum $\sup\{ d(x,y) : x,y \in A\}$, where $d(x,y)$ denotes the Euclidean distance between $x$ and $y$.  
\end{definition}
We begin by considering real-valued functions.

\begin{proposition*}
Let $f : \mathbb{R}^n \rightarrow \mathbb{R}$ be a continuous function where $n>1$.  Then for any $M > 0$, there exists $y\in \mathbb{R}$ whose fiber has diameter greater than $M$.
\label{m=1}
\end{proposition*}
\begin{proof}
Assume that some $M>0$ bounds all fiber diameters.  Consider three points $a, b, c \in \mathbb R^n$ such that the distance between any two is $2M$, as in Figure \ref{drawing}.  As $M$ bounds the fiber diameters, $f(a)$, $f(b)$, and $f(c)$ must be distinct; without loss of generality, suppose $f(a) < f(b) < f(c)$.  By the intermediate value theorem, the line segment $\overline{ac}$ contains a point $x$ such that $f(x) = f(b)$.  As the distance from $b$ to any point on $\overline{ac}$ is greater than $M$, the fiber containing $b$ must have diameter greater than $M$, contradicting our assumption that $M$ bounds all fiber diameters.
\end{proof}

\begin{figure}
\begin{center}
\includegraphics[width=0.23\linewidth]{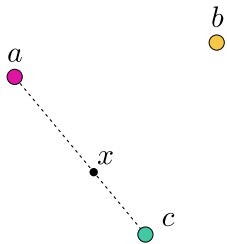}
\caption{Three points in $\mathbb{R}^n$ such that the distance between each pair is $2M$.  If $f(a)<f(b)<f(c)$, then by the intermediate value theorem there exists a point $x$ on the segment connecting $a$ to $c$ such that $f(a)<f(x)=f(b)<f(c)$, and so $x$ and $b$ belong to the same fiber.  Since the distance from $x$ to $b$ is greater than $M$, $M$ cannot bound all fiber diameters.}
\label{drawing}
\end{center}
\end{figure}

The intermediate value theorem plays a central role in the proof above, and will turn up again several times in what follows.  A generalization of this proposition can be established using the Borsuk-Ulam theorem \cite{borsuk1933drei}, a result about continuous mappings from an $n$-sphere $S^n$ to $\mathbb{R}^n$.
\begin{theorem*}[Borsuk-Ulam, 1933]
For every continuous map $f : S^n \rightarrow \mathbb{R}^n$ there exist antipodal points $x, -x \in S^n$ such that $f(x) = f(-x)$. 
\end{theorem*}
The Borsuk-Ulam theorem holds for $n$-dimensional spheres of any radius centered at the origin, and also for spheres centered at any point in a Euclidean space of dimension greater than $n$.  In the latter case, the antipodal map is the symmetry about the center of the sphere.  We can now show the following.
\begin{theorem}
Let $f : \mathbb{R}^n \rightarrow \mathbb{R}^m$ be a continuous map where $n>m$.  Then for any $M > 0$, there exists $y\in \mathbb{R}^m$ whose fiber has diameter greater than $M$.
\label{first theorem}
\end{theorem}
\begin{proof}
Consider an $m$-sphere $S^m \subset \mathbb{R}^n$ with radius $M$ and centered at the origin.  By Borsuk-Ulam, there are antipodal points $x, -x \in S^m$ such that $f(x) = f(-x)$.  The points $x$ and $-x$ lie in the same fiber of $f$ and $d(x, -x) = 2M$, so this fiber has diameter at least $2M$.
\end{proof}

An analogous result will hold for any domain $X \subseteq \mathbb R^n$ in which we can isometrically embed $m$-spheres of arbitrarily large diameter and any co-domain $Y \subseteq \mathbb R^m$, where $X$ and $Y$ are endowed with the subspace topology.

\section*{Small Fibers}
At this point we consider the union of all small fibers.  For the remainder of the paper, we consider an arbitrarily chosen but fixed $M>0$.
\begin{definition}
We call a nonempty fiber {\bf small} if its diameter is less than $M$.  
\end{definition}

\begin{namedtheorem}[Small fiber lemma]
Let $f : \mathbb{R}^n \rightarrow \mathbb{R}$ be a continuous map where $n>1$.  Given three points $a,b,c \in \mathbb{R}^n$ such that the distance between each pair is at least $M$, no more than two belong to small fibers of $f$.  
\label{thelemma}
\end{namedtheorem}
\begin{proof}
Assume that $a,b,c \in \mathbb{R}^n$ all belong to small fibers.  If the distance between each pair is at least $M$, then $f(a)$, $f(b)$, and $f(c)$ must be distinct; without loss of generality, suppose $f(a) < f(b) < f(c)$.  Since the complement in $\mathbb{R}^n$ of the open ball of radius $M$ centered at $b$ is path-connected, being homeomorphic to the product $S^{n-1} \times [M, \infty)$, a curve can be drawn from $a$ to $c$ such that the distance from $b$ to every point along the curve is at least $M$.  However, by the intermediate value theorem, for some point $x$ on the curve, $f(a) < f(x) = f(b) < f(c)$.  The fiber containing $b$ therefore has diameter at least $M$, and so cannot be small.    
\end{proof}

We now show that the small fiber lemma places restrictions on the union of small fibers for real-valued functions.
\begin{theorem}
Let $f : \mathbb{R}^n \rightarrow \mathbb{R}^m$ be a continuous map where $n>m$.  When $m=1$, the union of small fibers is bounded; when $m>1$ the union of small fibers can be unbounded.
\label{second theorem}
\end{theorem}
\begin{proof}
We begin with the case where $m=1$.  Recall that a fiber is small if its diameter is less than $M$.  If the union of all small fibers is contained in an open ball of radius $M$, then of course the union of small fibers is bounded.  If the union of all small fibers is not contained in an open ball of radius $M$, then there must exist points $a,b$ such that both belong to small fibers and such that $d(a,b)\geq M$.  It then follows from the small fiber lemma that for any point $x$ belonging to any small fiber, either $d(x,a)<M$ or else $d(x,b)<M$.  The set of all such points is of course bounded.

When $m>1$, a simple example shows that the union of small fibers of $f$ can be unbounded.  Consider the continuous map $f : \mathbb{R}^n \rightarrow \mathbb{R}^m$ given by:
\begin{equation}
f(x_1,x_2, \ldots, x_n) = (x_1, \sqrt{\Sigma_{i=2}^{n} {x_i^2}}, 0, 0, \ldots, 0).
\end{equation}
Note that fibers of $f$ are ($n-2$)-spheres with centers along the $x_1$ axis.  The union of all small fibers is the set of points whose distance to the $x_1$ axis is less than $M/2$.
\end{proof}

Although the union of small fibers is bounded when $m=1$, two small fibers can be located arbitrarily far apart.  Consider for example the Urysohn-like function $f : \mathbb{R}^n \rightarrow \mathbb{R}$, for $n>1$, given by
\begin{equation}
f(x) = \frac{d(x,a)^2}{d(x,a)^2 +d(x,b)^2},
\label{Urysohn-equation}
\end{equation}
for distinct points $a$ and $b$; an example in the case $n=2$ is illustrated in Figure \ref{threesets}.  With one exception, all fibers of $f$ are ($n-1$)-spheres whose centers lie on the line that passes through $a$ and $b$, but not on the segment $\overline{ab}$.  An additional fiber is the perpendicular bisector of $\overline{ab}$, a hyperplane that becomes a sphere when the point at infinity is adjoined.  The spherical fibers are smallest when their centers are closest to $\overline{ab}$, and grow as their centers move away from it.  The distance between the sets of small fibers can be made arbitrarily large by moving $a$ and $b$ arbitrarily far apart.

\begin{figure}[h]
\begin{center}
\includegraphics[width=0.55\linewidth]{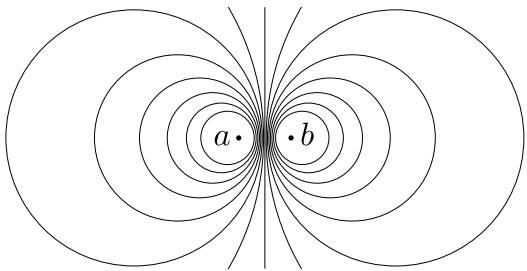}
\caption{Circular fibers of a continuous function $f : \mathbb{R}^2 \rightarrow \mathbb{R}$ given by (\ref{Urysohn-equation}). All small fibers are contained in two regions centered near $a$ and $b$; these regions can be located arbitrarily far apart.}
\label{threesets}
\end{center}
\end{figure}

\section*{Small Fibers and Boundedness}

Aside from illustrating that small fibers can be located arbitrarily far apart, Figure \ref{threesets} also highlights a general property of continuous real-valued functions with small fibers.  In particular, we can use an approach similar to that used in proving the small fiber lemma to show the following.
\begin{theorem}
Let $f : \mathbb{R}^n \rightarrow \mathbb{R}$ be a continuous function where $n>1$.  If $f$ has a small fiber, then $f$ is bounded from above or from below.  Further, if $a,b \in \mathbb{R}^n$ both belong to small fibers and $d(a,b)\geq M$, then $f$ is bounded from above and below. 
\label{third theorem}
\end{theorem}
\begin{proof}
If $f$ has a small fiber $f^{-1}(y)$, then this fiber is contained in some closed ball $B$ of radius $M$; since $B$ is closed, $f$ is bounded on it.  If $f$ is unbounded from above and from below, then there exist points $a, b \not\in B$ such that $f(a) < y < f(b)$.  As in the proof of the small fiber lemma, since the complement of $B$ in $\mathbb{R}^n$ is path-connected, a curve can be drawn in it from $a$ to $b$.  By the intermediate value theorem there exists a point $x$ on that curve such that $y = f(x)$, and so the fiber over $y$ is not small, contradicting our assumption.

The second part of the theorem can be proven in a similar manner.  If $a,b \in \mathbb{R}^n$ both belong to small fibers and $d(a,b)\geq M$, then $f(a)$ and $f(b)$ must be distinct; without loss of generality, suppose $f(a) < f(b)$.  Consider the closed ball $B$ of radius $M$ centered at $b$.  Since $B$ is closed, $f$ is bounded on it.  If $f$ is unbounded from above, then there exists a point $c \not\in B$ such that $f(a) < f(b) < f(c)$.  As above, a curve can be drawn from $a$ to $c$ such that the distance from $b$ to every point along the curve is at least $M$.  By the intermediate value theorem, there exists a point $x$ along that curve such that $f(x)=f(b)$.  Since $d(x,b)\geq M$, $b$ cannot belong to a small fiber, contradicting our assumption and showing that $f$ must be bounded from above. An analogous argument shows that $f$ is bounded from below.
\end{proof}

Note that discontinuous functions, even those that are bounded, integrable, and decay to zero, need not have fibers of arbitrarily large diameter.  Consider, for example, the following function $f : \mathbb{R}^2 \rightarrow \mathbb{R}$:
\begin{align}
f(x,y) &= 
  \begin{cases}
    \frac{1}{2}^{\lfloor x \rfloor}\frac{1}{3}^{\lfloor y \rfloor} & x\geq 0, y\geq 0\\
    \frac{1}{5}^{- \lfloor x \rfloor}\frac{1}{7}^{\lfloor y \rfloor} & x < 0, y\geq 0\\
    \frac{1}{11}^{\lfloor x \rfloor}\frac{1}{13}^{-\lfloor y \rfloor} & x\geq 0, y < 0\\
    \frac{1}{17}^{-\lfloor x \rfloor}\frac{1}{19}^{-\lfloor y \rfloor} & x < 0, y < 0.
  \end{cases}
\label{Lp}
\end{align}
Here $f$ takes a unique rational value on each unit square in the plane, and so every nonempty fiber of $f$ has diameter $\sqrt{2}$.  Note that $f$ tends to 0 as $x^2+y^2$ tends to $\infty$.  Moreover, $f$ is in $L^1$ and $L^{\infty}$, and hence in $L^p$ for all $1 \leq p \leq \infty$.  This example can be generalized for arbitrary $n>m$, with a suitable choice of prime numbers.

\section*{Conclusions}

The analysis above provides a cautionary tale for data science analysts.  The use of continuous maps to reduce the dimension of point sets in high-dimensional Euclidean spaces entails what we might call the ``curse of continuity'' -- there will exist points arbitrarily far apart that are identified under such maps.  Not only will knowledge of $f(x)$ not allow us to recover $x$ exactly, but we will generally be unable to determine $x$ to within any finite error.  Under suitable restriction of the domain this issue might be avoided, but knowledge of such restrictions is not always available a priori.  

In contrast, discontinuous mappings suffer no such inherent limitations.  Equation (\ref{Lp}), for example, can be scaled such that its fibers are $n$-dimensional cubes of edge length $\epsilon$.  The diameter of each fiber is then precisely $\epsilon\sqrt{n}$.  Knowledge of $f(x)$ then allows us to determine $x$ to within a maximal error $\epsilon\sqrt{n}$.  

An application highlighting some limitations of continuous maps in analyzing structure in large point sets can be found in \cite{lazar2015}.  In computational materials science research, continuous ``order parameter'' mappings are often constructed to summarize structural information near each particle in a system of particles.  This order parameter is subsequently used to identify larger-scale structural features of the system.  The continuity of the order-parameter entails that points arbitrarily far apart in a relevant configuration space will map to the same order-parameter value.  Consequently, continuous order-parameters regularly fail to distinguish structurally distinct configurations of points, making automated analysis difficult or impossible.  In that paper, the authors suggest a discrete order-parameter, based on Voronoi cell topology, which largely avoids this degeneracy.

We note that Theorem \ref{first theorem} can be obtained as a consequence of Corollary 0.3 in \cite{calegari2000degree}, though the proof here is simpler.  Also, a similar result for proper mappings can be found as a consequence of an exercise presented at the end of Section 3.3 in \cite{bowditch2006course}.  Finally, Theorem \ref{first theorem} can also be obtained as a simple corollary of what Larry Guth has called the ``Large fiber lemma'' \cite[Section 7.6]{burago2014few}, \cite[Section 6]{guth2010metaphors}, itself a corollary of the Lebesgue covering lemma which is used in topological dimension theory.  
\begin{namedtheorem}[Large fiber lemma]
If $f : [0,1]^n \rightarrow \mathbb{R}^m$ is a continuous map where $n>m$, then one of the fibers of $f$ has diameter at least 1.  
\label{largefiberlemma}
\end{namedtheorem}
The technique used in the proof of Theorem \ref{first theorem} can also be used to strengthen this lemma in two ways.  First, we can replace the unit $n$-cube with the inscribed $(n-1)$-sphere of unit diameter.  Second, the inequality of the conclusion becomes an equality.  

\section*{Questions}
This paper is concerned with fibers in $\mathbb{R}^n$ equipped with the standard Euclidean metric.  It is easily shown that Theorem \ref{first theorem} can fail when $\mathbb{R}^n$ is equipped with other metrics.  For example, we can map $\mathbb{R}^2$ homeomorphically to the cylinder capped by a half-sphere on one end and unbounded on the other.  If the origin is mapped to the center of the spherical cap, then fibers of the distance function to the origin are circles whose diameters are all bounded by that of the cylinder.  Thus, one might wonder about necessary and sufficient conditions on the metric for Theorem \ref{first theorem} to hold.

In this paper, we proved that the union of small fibers can have unbounded diameter when $m>1$.  However, one might ask if there must exist a line through the origin in $\mathbb{R}^n$ onto which the projection of this set is bounded.  If so, is there a lower bound on the number of such orthogonal lines, depending on $n$ and $m$?

\section*{Acknowledgements}  The authors would like to thank the anonymous referee for numerous helpful remarks on an earlier draft.

\end{document}